\documentclass[a4paper,twoside,reqno]{amsart}

\usepackage[pagewise]{lineno}
\usepackage[utf8]{inputenc}
\usepackage[left=3cm, right=3cm, top=2.1cm, bottom=3cm]{geometry}
\usepackage{authblk}
\usepackage{amsmath,amsthm,amssymb,bbm,mathrsfs,graphicx,subfigure,sidecap,wrapfig,float,caption}
\usepackage[dvipsnames]{xcolor}
\usepackage[english]{babel}
\usepackage{fancyhdr}
\makeatletter
\@namedef{subjclassname@1991}{2020 Mathematics Subject Classification}
\makeatother

\title{Probability of Consensus of Hegselmann-Krause Dynamics}

\newtheorem{theorem}{Theorem}
\newtheorem{lemma}{Lemma}
\newtheorem{corollary}{Corollary}

\DeclareMathOperator{\uniform}{Uniform \,}

\begin{document}
\allowdisplaybreaks

\thispagestyle{firstpage}
\maketitle
\begin{center}
    Hsin-Lun Li\\
    School of Mathematical and Statistical Sciences,\\ Arizona State University, Tempe, AZ 85287, USA
\end{center}
\begin{abstract}
    The original Hegselmann-Krause (HK) model comprises a set of $n$ agents characterized by their opinion, a number in $[0,1]$. Agent $i$ updates its opinion $x_i$ via taking the average opinion of its neighbors whose opinion differs by at most $\epsilon$ from $x_i$. In the article, the opinion space is extended to $\mathbf{R^d}.$ The main result is to derive bounds for the probability of consensus. In general, we have a positive lower bound for the probability of consensus and demonstrate a lower bound for the probability of consensus on a unit cube. In particular for one dimensional case, we derive an upper bound and a better lower bound for the probability of consensus and demonstrate them on a unit interval. 
\end{abstract}

\section{Introduction}
The original Hegselmann-Krause (HK) model consists of a set of~$n$ agents characterized by their opinion, a number in~$[0,1]$. Agent $i$ updates its opinion $x_i$ via taking the average opinion of its neighbors whose opinion differs by at most~$\epsilon$ from~$x_i$ for a confidence bound~$\epsilon>0.$ In this essay, the opinion space is extended to~$\mathbf{R^d}$. The aim is to derive a lower bound for the probability of consensus for the synchronous HK model as follows:
\begin{equation}\label{HK}
 x(t+1)=A(t)x(t)\hbox{ for }t\in\mathbf{N},
\end{equation}
$$\begin{array}{c}
     \displaystyle  A_{ij}(t)=\mathbbm{1}\{j\in N_i(t)\}/|N_i(t)| \hbox{ and }
     \displaystyle x(t)=(x_1(t),\ldots,x_n(t))'=\hbox{transpose of }(x_1(t),\ldots,x_n(t))
\end{array}$$  for $[n]:=\{1,2,\ldots,n\}$, $N_i(t)=\{j\in[n]:\|x_j(t)-x_i(t)\|\leq\epsilon \}$ the collection of agent $i's$ neighbors at time $t$ and $\|\ \,\|$ the Euclidean norm. \cite{p5} gives an overview of HK models. \cite{p3, p4} elaborate that \eqref{HK} has finite-time convergence property. \cite{p2} further illustrates that the \emph{termination time} $$T_n=\inf\{t\geq0:x(t)=x(s)\text{ for all }s\geq t\}$$ is bounded from above. Finite-time convergence property is enough to imply $$\lim_{t\rightarrow\infty}\max_{i,j\in [n]}\|x_i(t)-x_j(t)\| \hbox{ exists}.$$ Let the initial opinions $x_i(0)$ be independent and identically distributed random variables with a convex support $S\subset\mathbf{R^d}$ of positive Lebesgue measure and a probability density function $f$, where $P(x_i\in B)=\int_{B}f(x_i)dm(x_i)$ for all $i\in[n]$, $B$ a Borel set and $m$ the Lebesgue measure. Here, we say a function or a set is measurable if it is Lebesgue measurable.  A \emph{profile} at time~$t$ is an undirected graph~$\mathscr{G}(t)=(\mathscr{V}(t),\mathscr{E}(t))$ with the vertex set and edge set $$\mathscr{V}(t)=[n]\hbox{ and } \mathscr{E}(t)=\{(i,j):i\neq j\hbox{ and }\|x_i(t)-x_j(t)\|\leq\epsilon\}.$$ A profile $\mathscr{G}(t)$ is \emph{$\delta$-trivial} if any two vertices are at a distance of at most~$\delta$ apart. Observe that a consensus is reached at time $t+1$ if~$\mathscr{G}(t)$ is~$\epsilon$-trivial. 

\section{Main results}
Define $$\mathscr{C}=\{\lim_{t\rightarrow\infty}\max_{i,j\in [n]}\|x_i(t)-x_j(t)\|=0\},$$ the collection of all sample points that lead to a consensus.
\begin{theorem}\label{pT1}
$$P(\mathscr{C})\geq P(\mathscr{G}(0)\hbox{ is connected})\hbox{ for }1\leq n\leq 4.$$
In general,
\begin{align*}
    P(\mathscr{C})&\geq P(\mathscr{G}(0)\text{ is }\epsilon\text{-trivial})\\
    &\geq P(x_i(0)\in B(x_1(0),\epsilon/2)\text{ for all }i\in[n])\\
    &=\int_{\mathbf{R^d}}f(x_1)\left(\int_{B(x_1,\epsilon/2)}f(x)dm(x)\right)^{n-1}dm(x_1)>0\text{ for }n\geq1.
\end{align*}
In particular, the probability of consensus is positive.
 
\end{theorem}

\begin{corollary}\label{pCo1}
Assume that $S=[0,1]^d$ and $x_i(0)=\uniform([0,1]^d)$. Then, $$P(\mathscr{C})\geq \left((\frac{\epsilon}{2})^d m(B(0,1))\right)^{n-1}(1-\epsilon)^d=\left((\frac{\epsilon}{2})^d \frac{\pi^\frac{d}{2}}{\Gamma(\frac{d}{2}+1)}\right)^{n-1}(1-\epsilon)^d$$ for all $i\in[n]$ and $\epsilon\in(0,1).$
\end{corollary}
Define $$(1)=\arg\min_{k\in[n]}x_k\text{ and }(i)=\arg\min_{k\in[n]-\{(j)\}_{j=1}^{i-1}}x_k\text{ for }i\geq2.$$ Namely $x_{(i)}$ is the $i$-th smallest number among $(x_k)_{k=1}^n.$ For $n\geq 4$, let $m=\lfloor\frac{n-4}{3}\rfloor$ and $k=n-m-1.$ Say $\mathscr{G}(t)$ satisfies $(\ast)$ if $$((m+2),(k))\in \mathscr{E}(t)\text{ and }(x_{(n)}-x_{(k)}+x_{(m+2)}-x_{(1)})(t)\leq\epsilon.$$
Say $\mathscr{G}(t)$ satisfies $(\ast\ast)$ if $$\max\left((x_{(n)}-x_{(n-i-1)})(t),\ (x_{(n-i-1)}-x_{(i+2)})(t),\ (x_{(i+2)}-x_{(1)})(t)\right)\leq\frac{\epsilon}{2}\text{ for some }0\leq i\leq m.$$
\begin{theorem}[$d=1$]\label{pT2}
$$P(\mathscr{C})=P(\mathscr{G}(0)\text{ is connected})\text{ for }1\leq n\leq 4.$$
$$P(\mathscr{C})\geq P(\mathscr{G}(0)\text{ satisfies }(\ast))\hbox{  for }5\leq n\leq7.$$
In general, 
\begin{align*}
  &P(\mathscr{G}(0)\text{ is connected})\geq P(\mathscr{C})\geq P(\mathscr{G}(0)\text{ is }\epsilon\text{-trivial or satisfies }(\ast\ast))\text{ for }n\geq 1.
\end{align*}

\end{theorem}

\begin{corollary}\label{pCo2}
Let $S=[0,1]$, $d=1$, $\epsilon\in(0,1)$ and $x_i(0)=\uniform([0,1])$ for all $i\in[n]$. Then,\\
for $n=2,$ 
$$P(\mathscr{C})=\epsilon(2-\epsilon)$$
for $n=3,$
$$P(\mathscr{C})=\left\{
\begin{array}{cc}

 6\epsilon^2(1-\epsilon)&\epsilon\in(0,\frac{1}{2})\\
1-2(1-\epsilon)^3&\epsilon\in[\frac{1}{2},1)
\end{array}\right.$$
for $n=4,$
$$P(\mathscr{C})=\left\{\begin{array}{cc}
   24\epsilon^3(1-3\epsilon)+36\epsilon^4  & \epsilon\in(0,\frac{1}{3}) \\\\
    19\epsilon^4-4\epsilon^3(1-2\epsilon)+(1-2\epsilon)^4-6\epsilon^2(3\epsilon-1)^2\\
    -4\epsilon(1-2\epsilon)^3+12\epsilon^3(1-2\epsilon)+12\epsilon^2(1-2\epsilon)^2 & \epsilon\in[\frac{1}{3},\frac{1}{2})\\\\
  \epsilon^4+4\epsilon^3(1-\epsilon)+6\epsilon^2(1-\epsilon)^2+4\epsilon(1-\epsilon)^3-2(1-\epsilon)^4   &\epsilon\in[\frac{1}{2},1)
\end{array}\right.$$
for $n\geq1,$
\begin{align*}
    &P(\mathscr{G}(0)\text{ is }\epsilon\text{-trivial})=\epsilon^{n-1}[n-(n-1)\epsilon]\\
    &P(x_i(0)\in B(x_1(0),\epsilon/2)\text{ for all }i\in[n])=\frac{2}{n}\epsilon^n(1-\frac{1}{2^n})+\epsilon^{n-1}(1-\epsilon).
\end{align*}
In general,$$P(\mathscr{C})\geq P(\mathscr{G}(0)\text{ is }\epsilon\text{-trivial})=\epsilon^{n-1}[n-(n-1)\epsilon]\text{ for }n\geq1.$$
\end{corollary}

\begin{minipage}[c]{\textwidth}
\centering
\includegraphics[width=.3\textwidth]{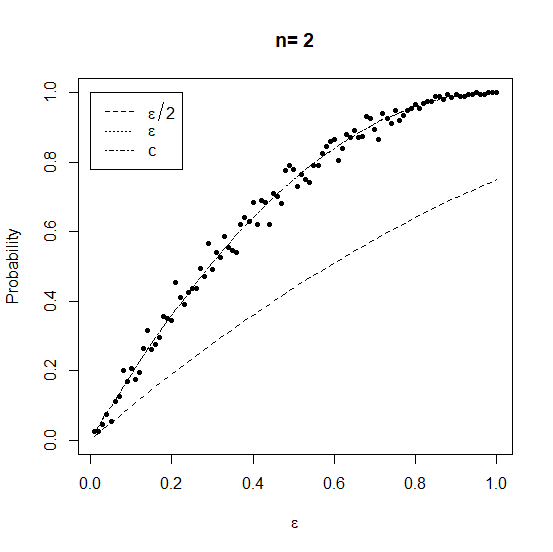}
\includegraphics[width=.3\textwidth]{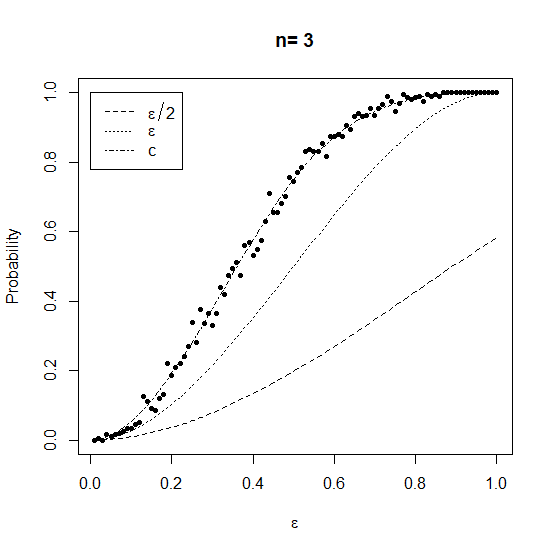}
\includegraphics[width=.3\textwidth]{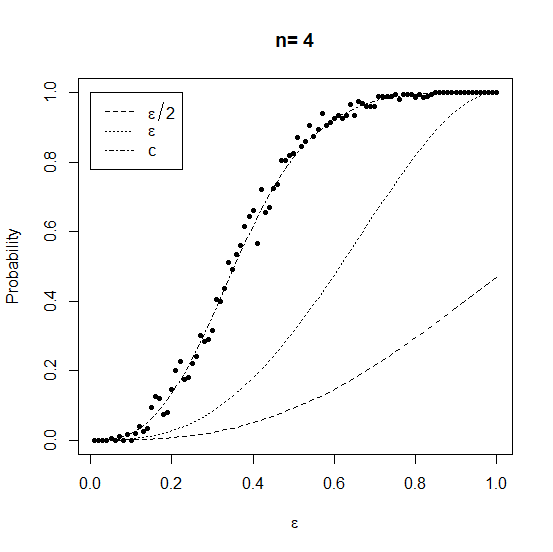}
\includegraphics[width=.45\textwidth]{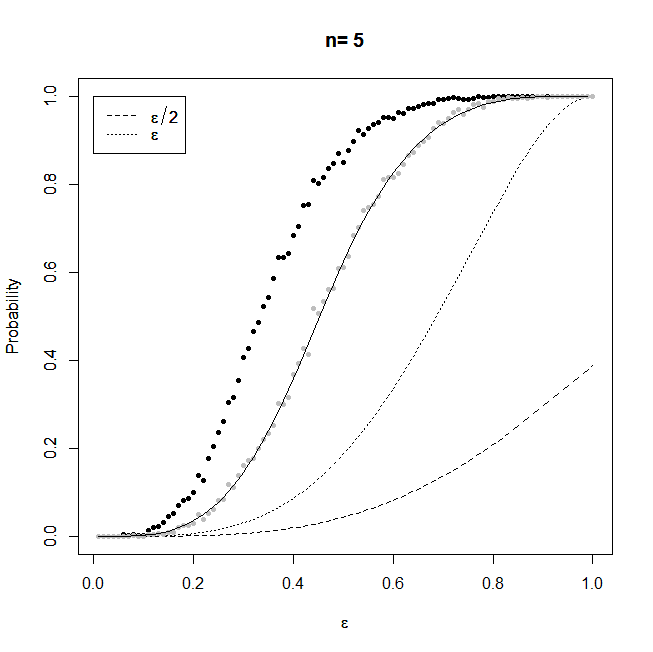}
\includegraphics[width=.45\textwidth]{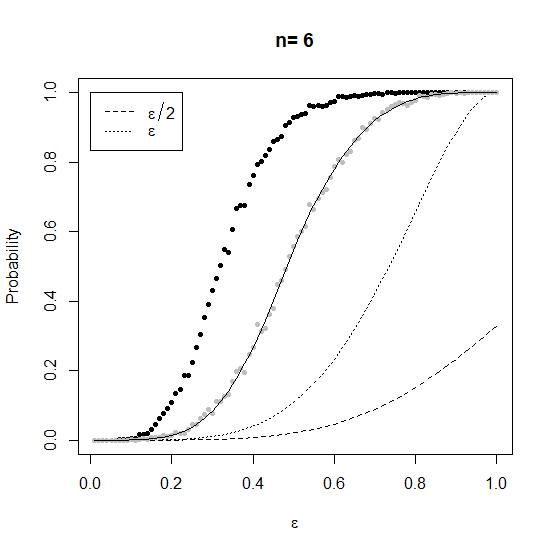}
\includegraphics[width=.45\textwidth]{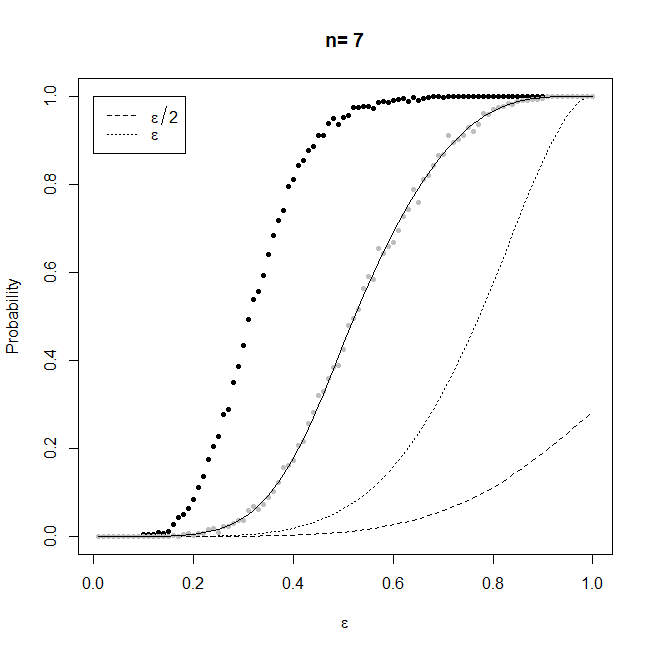}
\includegraphics[width=.45\textwidth]{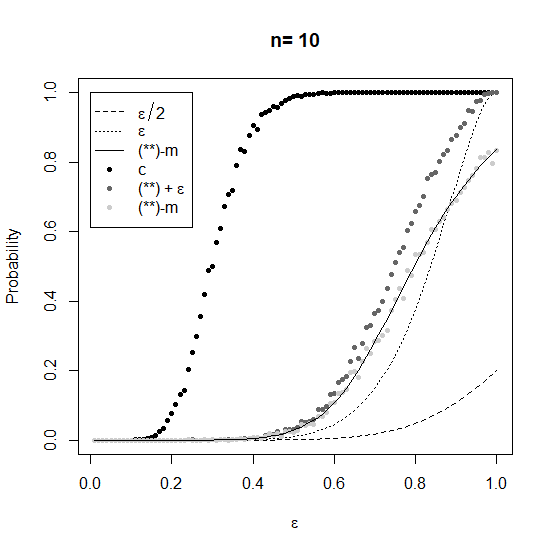}
\captionof{figure}{Bounds for $P(\mathscr{C})$}
\label{pg1}
\end{minipage}
From Figure \ref{pg1}, labels $\epsilon/2,$ $\epsilon$ and $c$ denote respectively the lines of $$P(x_i(0)\in B(x_1(0),\epsilon/2)\text{ for all }i\in[n]),\ P(\mathscr{G}(0)\text{ is }\epsilon\text{-trivial})\hbox{ and }P(\mathscr{G}(0)\text{ is connected}).$$ The black points are simulations for the probability of consensus. For $n=2,$ $$P(\mathscr{C})=P(\mathscr{G}(0)\text{ is connected})=P(\mathscr{G}(0)\text{ is }\epsilon\text{-trivial})$$ so $\epsilon$ and $c$-lines overlap. Observe that the points for the probability of consensus are around $c$-line for $2\leq n\leq 4$, which meets the theory. For $5\leq n\leq7,$ the gray points and the solid line are respectively simulations and numerical integrals of $P(\mathscr{G}(0)\text{ satisfies }(\ast))$, suggesting that theoretically $P(\mathscr{G}(0)\text{ satisfies }(\ast))$ is a better lower bound for $P(\mathscr{C})$ than $P(\mathscr{G}(0)\text{ is }\epsilon\text{-trivial})$. For $n=10,$ the dark gray points are simulations of $P(\mathscr{G}(0)\text{ is }\epsilon\text{-trivial or satisfies }(\ast\ast)),$ and the points and solid line are respectively simulations and numerical integrals of $P(\mathscr{G}(0)\text{ satisfies }(\ast\ast)\text{ and }i=(m)).$ Suggest that $$P(\mathscr{G}(0)\text{ is }\epsilon\text{-trivial})\vee P(\mathscr{G}(0)\hbox{ satisfies }(\ast\ast)\text{ and }i=(m))$$ is a better lower bound than each of the two for the probability of consensus.

\section{Probability of consensus}
To derive a better lower bound for the probability of consensus, we study properties other than~$\epsilon$-triviality that leads to a consensus. If a profile is connected-preserving, then a consensus can be achieved in finite time. We illustrate that any profile~$\mathscr{G}$ is connected-preserving for~$1\leq n\leq4$ and some profile~$\mathscr{G}$ of some configuration~$x$ fails to remain connected for~$n>4.$ Thus $P(\mathscr{C})\geq P(\mathscr{G}(0)\text{ is connected})$ for $1\leq n\leq4.$ It is straightforward in general, $P(\mathscr{G}(0)\text{ is }\epsilon\text{-trivial})$ is a lower bound for $P(\mathscr{C})$ but it is uneasy to calculate in high dimensions. Therefore we provide an easier calculated lower bound for the probability of consensus and also depict that the probability of consensus is positive.

 Lemma \ref{pL1} is the key to depict that any profile~$\mathscr{G}$ is connected-preserving for~$1\leq n\leq4.$
\begin{lemma}[\cite{p1}]\label{pL1}
Given $\lambda_1,...,\lambda_n$ in $\mathbf{R}$ with $\sum_{i=1}^n\lambda_i=0$ and $x_1,...,x_n$ in $\mathbf{R^d}$. Then for $\lambda_1x_1+\lambda_2x_2+...+\lambda_nx_n,$ the terms with positive coefficients can be matched with the terms with negative coefficients in the sense that
\begin{align*}
     \sum_{i=1}^n\lambda_ix_i=\sum_{i,c_i\geq0,j,k\in[n]} c_{i}(x_{j}-x_{k})\text{ and }\sum_i c_i=\sum_{j,\lambda_j\geq0}\lambda_j.
\end{align*}
\end{lemma}

 From Lemma~\ref{pL1}, we derive a good upper bound for~$\|x_i(t+1)-x_j(t+1)\|$ for any~$(i,j)\in \mathscr{E}(t).$
\begin{lemma}\label{pL2}
Assume that $(i,j)\in \mathscr{E}(t)$ and that $|N_i(t)|\leq|N_j(t)|$. Then, 
\begin{align*}
    &\|x_i(t+1)-x_j(t+1)\|\leq\epsilon\left(3-|N_i(t)\cap N_j(t)|(\frac{2}{|N_j(t)|}+\frac{1}{|N_i(t)|})\right).
\end{align*}
\end{lemma}
\begin{proof}
Let $x=x(t)$, $x'=x(t+1)$ and $N_i=N_i(t)$ for any $i\in[n]$. Via Lemma \ref{pL1}, for any $i,j\in[n]$,
\begin{align*}
    &x_i'-x_j'=\frac{1}{|N_i|}\sum_{k\in N_i}x_k-\frac{1}{|N_j|}\sum_{k\in N_j}x_k\\
    &=(\frac{1}{|N_i|}-\frac{1}{|N_j|})\sum_{k\in N_i\cap N_j}x_k+\frac{1}{|N_i|}\sum_{k\in N_i-N_j}x_k-\frac{1}{|N_j|}\sum_{k\in N_j-N_i}x_k\\
    &=\sum_{p\in N_i\cap N_j,q\in N_j-N_i}a_r(x_p-x_q)+\sum_{p\in N_i-N_j,q\in N_j-N_i}b_r(x_p-x_q)
\end{align*}
where $a_r$, $b_r\geq0$, $\sum_r a_r=(\frac{1}{|N_i|}-\frac{1}{|N_j|})|N_i\cap N_j|$ and $\sum_r b_r=|N_i-N_j|/|N_i|.$ Thus by the triangle inequality,
\begin{align*}
    &\|x_i'-x_j'\|\leq\sum_{p\in N_i\cap N_j,q\in N_j-N_i}a_r(\|x_p-x_j\|+\|x_j-x_q\|)\\
    &+\sum_{p\in N_i-N_j,q\in N_j-N_i}b_r(\|x_p-x_i\|+\|x_i-x_j\|+\|x_j-x_q\|)\\
    &\leq(\frac{1}{|N_i|}-\frac{1}{|N_j|})|N_i\cap N_j|(\epsilon+\epsilon)+\frac{|N_i-N_j|}{|N_i|}(\epsilon+\epsilon+\epsilon)\\
    &=(\frac{1}{|N_i|}-\frac{1}{|N_j|})|N_i\cap N_j|2\epsilon+(1-\frac{|N_i\cap N_j|}{|N_i|})3\epsilon\\
    &=\epsilon\left(3-|N_i\cap N_j|(\frac{2}{|N_j|}+\frac{1}{|N_i|})\right).
\end{align*}

\end{proof}
 Thus $\epsilon\left(3-|N_i(t)\cap N_j(t)|(\frac{2}{|N_j(t)|}+\frac{1}{|N_i(t)|})\right)\leq\epsilon$ implies $(i,j)\in \mathscr{E}(t+1).$ For the following lemmas, assume $x=x(t)$, $x'=x(t+1)$ and $N_i=N_i(t)$ for any $i\in[n]$ without specifying.
\begin{lemma}\label{pL3}
Assume that $(i,j)\in \mathscr{E}(t)$ with $|N_i(t)|\leq |N_j(t)|$ and that $$|N_i(t)\cap N_j(t)|(\frac{2}{|N_j(t)|}+\frac{1}{|N_i(t)|})\geq 2.$$ Then, $(i,j)\in \mathscr{E}(t+1).$ 
\end{lemma}
\begin{proof}
By Lemma \ref{pL2},
\begin{align*}
    &\|x_i'-x_j'\|\leq\epsilon\left(3-|N_i\cap N_j|(\frac{2}{|N_j|}+\frac{1}{|N_i|})\right)\leq\epsilon(3-2)=\epsilon.
\end{align*}
Thus $(i,j)\in \mathscr{E}(t+1).$
\end{proof}
 Observe that $|N_i(t)\cap N_j(t)|\geq2$ for $(i,j)\in \mathscr{E}(t)$ and that the inequality $\frac{2}{|N_j(t)|}+\frac{1}{|N_i(t)|}\geq1$ automatically holds for $1\leq n\leq3.$ It is not straight forward to see Lemma \ref{pL3} works for $n=4.$ However, categorizing the degrees of the pair $(i,j)\in \mathscr{E}(t)$, a profile $\mathscr{G}$ remains connected for~$n=4.$
\begin{lemma}[connected-preserving]\label{pL4}
For $1\leq n\leq4,$ a profile~$\mathscr{G}$ is connected-preserving. 
\end{lemma}
\begin{proof}
Since $i,j\in N_i\cap N_j$ for $(i,j)\in \mathscr{E}(t)$ and $|N_i|\leq n$ for all $i\in [n],$ 
$$|N_i\cap N_j|(\frac{2}{|N_j|}+\frac{1}{|N_i|})\geq2(\frac{2}{n}+\frac{1}{n})\geq2(\frac{2}{3}+\frac{1}{3})=2\text{ for }1\leq n\leq 3.$$ From Lemma \ref{pL3}, $(i,j)\in \mathscr{E}(t+1).$ Thus any edge in $\mathscr{E}(t)$ remains in $\mathscr{E}(t+1).$ Hence a profile is connected-preserving for~$1\leq n\leq 3$.\vspace*{4pt}\\
For $n=4,$ let $d_i=d_i(t)$=the degree of vertex $i$ at time $t$. For $(i,j)\in\mathscr{E}(t)$ and $d_i\leq d_j,$ $i$ is either a leaf or a non-leaf, and $i$ and $j$ can not be both leaves. So the cases of $(d_i,d_j)$ are as follows:
$$\begin{bmatrix}
d_i&1&1&2&2&3\\
d_j&2&3&2&3&3
\end{bmatrix}.$$
Thus the cases of corresponding $(|N_i|,|N_j|)$ are
$$\begin{bmatrix}
|N_i|&2&2&3&3&4\\
|N_j|&3&4&3&4&4
\end{bmatrix}.$$
From Lemma \ref{pL3}, if $\frac{2}{|N_j|}+\frac{1}{|N_i|}\geq1$ or $|N_i\cap N_j|(\frac{2}{|N_j|}+\frac{1}{|N_i|})\geq2$, then $(i,j)\in \mathscr{E}(t+1).$ We check if each case meets one of the two conditions:
\begin{align*}
  (2,3):\ &\frac{2}{3}+\frac{1}{2}>\frac{1}{2}+\frac{1}{2}=1\\
   (2,4):\ &\frac{2}{4}+\frac{1}{2}=\frac{1}{2}+\frac{1}{2}=1\\
   (3,3):\ &\frac{2}{3}+\frac{1}{3}=1\\
  (3,4):\ &3(\frac{2}{4}+\frac{1}{3})=\frac{3}{2}+1>1+1=2\\
  (4,4):\ &4(\frac{2}{4}+\frac{1}{4})=2+1=3>2.
\end{align*}
Since each case satisfies one of the two conditions, $(i,j)\in \mathscr{E}(t+1)$ for each case above, so a profile is connected-preserving for $n=4$.
\end{proof}
 Lemma \ref{pL3} does not work for $n=5$ even by categorizing the degrees of the pair $(i,j)\in \mathscr{E}(t)$. But indeed some profile $\mathscr{G}$ of some configuration $x$ fails to remain connected. 
\begin{figure}[H]
    \centering
  \subfigure[]{\label{pg2}
  \includegraphics[width=.3\textwidth]{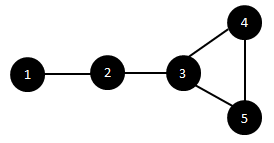}}
  \subfigure[]{\label{pg3}
  \includegraphics[width=.3\textwidth]{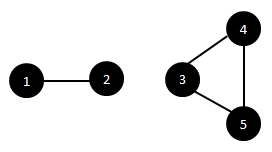}}
  \captionof{figure}{}
\end{figure}
\begin{lemma}\label{pL5}
For $n\geq5,$ some profile~$\mathscr{G}$ of some configuration~$x$ is not connected-preserving.
\end{lemma}
\begin{proof}
Need only show that there is a configuration~$x$ with the profile~$\mathscr{G}(0)$ connected but $\mathscr{G}(1)$ disconnected for $n=5$ and $d=1$. Consider $$\epsilon=1,\ x_1(0)=-1,\ x_2(0)=0,\ x_3(0)=1\ \hbox{and}\ x_4(0)=x_5(0)=2.$$ Then, $\mathscr{G}(0)$ as Figure \ref{pg2} is connected, $$x_1(1)=-0.5,\ x_2(1)=0,\ x_3(1)=1.25\ \hbox{and}\ x_4(1)=x_5(1)=\frac{5}{3}.$$ So $\mathscr{G}(1)$ as Figure \ref{pg3} is disconnected. For $n>5,$ let the new added vertices whose opinion be -1 or 2. Then, at the next time step, opinion 0 goes much closer to -1 or opinion 1 goes much closer to 2, and so $\mathscr{G}(0)$ connected but $\mathscr{G}(1)$ disconnected. This completes the proof.
\end{proof}
 Hence a profile is connected-preserving for $n\leq4$, and some profile $\mathscr{G}$ of some configuration $x$ fails to remain connected for $n>4.$ We can estimate the probability of consensus via the initial profiles $\mathscr{G}(0).$ Lemmas \ref{pL6}-\ref{pL9} indicate the probability of consensus is positive. 

\begin{lemma}\label{pL6}
Let $g>0$ be a measurable function on a measurable set $A$ with $m(A)>0$. Then, $\int_A gdm>0.$  
\end{lemma}
\begin{proof}
Let $E_k=\{g>\frac{1}{k}\}$. Then, $A=\cup_{k\geq1}E_k.$ Suppose by contradiction that $\int_A gdm=0$. Then, $$\frac{1}{k}m(E_k)\leq\int_{E_k}gdm\leq\int_A gdm=0\text{ for all }k\geq1.$$ Thus via the subadditivity of a measure,
$$m(A)\leq\sum_{k\geq1}m(E_k)=0,\text{ a contradiction}.$$ 
\end{proof}

\begin{lemma}\label{pL7}\rm{(i)} The intersection of convex sets is convex. \rm{(ii)} The closure of a convex set is convex.
\end{lemma}
\begin{proof}\rm{(i)} Let $Q_{\alpha}'s$ be convex sets. If $\cap_{\alpha}Q_{\alpha}=\emptyset$, then clearly it is convex. Else, for $a,b\in\cap_{\alpha}Q_{\alpha},a,b\in Q_{\alpha}$ for all $\alpha.$ So by convexity of convex sets, any point on the segment $\overline{\rm ab}$ is in $Q_{\alpha}$ for all $\alpha.$ Thus any point on the segment $\overline{\rm ab}$ is in $\cap_{\alpha}Q_{\alpha}.$ \\
\rm{(ii)} Let $V$ be a convex set. For $v\in\overline{V}$, there exists $(v_n)_{n\geq1}\subset V$ with $v_n\rightarrow v$ as $n\rightarrow\infty.$ For $u,v\in\overline{V},$ $$tu+(1-t)v=\lim_{n,m\rightarrow\infty}[t u_n+(1-t) v_m]$$ where $u_n,v_m\in V$ for all $n,m\geq1$ and $t\in(0,1)$. By convexity of $V,\ t u_n+(1-t) v_m\in V$, so $tu+(1-t)v\in\overline{V}.$ 
\end{proof}
 A \emph{convex hull} generated by $v_1,\ldots,v_k\in\mathbf{R^d}$, denoted by $C(\{v_1,\ldots,v_k\})$, is the smallest convex set containing $v_1,\ldots,v_k$, $i.e.,\ C(\{v_1,\ldots,v_k\})=\{v:v=\sum_{i=1}^k a_iv_i,(a_i)_{i=1}^k\text{ is stochastic}\}$.
\begin{lemma}\label{pL8}
A convex set in $\mathbf{R^d}$ is measurable.
\end{lemma}
\begin{proof}
Let $\mathscr{L}$ be the collection of all Lebesgue sets in $\mathbf{R^d}$ and $V\subset\mathbf{R^d}$ be a convex set.  $$\text{Claim: }m(\partial{V})=0.$$
For $V^\circ=\emptyset$, if $m(\overline{V})>0$, then $\overline{V}$ is uncountable, and there exist $d+1$ distinct points, $v_1,v_2,...,v_{d+1}$, in $\overline{V}$ not in any hyperplane in $\mathbf{R^d}$. By convexity of $\overline{V},$ $C(\{v_1,...,v_{d+1}\})\subseteq\overline{V}$ with its interior nonempty, a contradiction.\\\\
For $V^\circ\neq\emptyset,$ since measurability is shift-invariant, may assume zero vector $\Vec{0}\in V^\circ.$ Then $B(0,r)\subset V$ for some $0<r<1$. For $n\in\mathbf{Z^+},$ let $A_n=B(0,n)\cap V$ then by Lemma \ref{pL7}, $A_n$ is bounded and convex. For $q\in\partial A_n,$ by convexity of $\overline{V}$ and $A_n\supset B(0,r),$ $$p=s\cdot q+(1-s)\cdot \Vec{0}\in A_n^\circ\text{ for all }s\in(0,1). $$ Thus $q\in \frac{1}{s}A_n^\circ.$ Since $\frac{1}{s}A_n^\circ\supset A_n^\circ,$
\begin{align*}
    &m(\partial A_n)\leq m(\frac{1}{s}A_n^\circ-A_n^\circ)=m(\frac{1}{s}A_n^\circ)-m(A_n^\circ)\\
    &=(\frac{1}{s})^d m(A_n^\circ)-m(A_n^\circ)\rightarrow 0\text{ as }s\rightarrow 1.
\end{align*}
Since $\cup_{n\geq 1}\partial A_n\supset \partial V,$ 
$$m(\partial V)\leq m(\cup_{n\geq 1}\partial A_n)\leq \sum_{n\geq 1}m(\partial A_n)=0.$$ Thus $\partial V$ is a null set. By the completion of Lebesgue measure $\partial V\cap V\in\mathscr{L}.$ Hence $$V=V^\circ\cap(\partial V\cap V)\in \mathscr{L}.$$ 
\end{proof}

\begin{lemma}\label{pL9}
Let $V\subset\mathbf{R^d}$ be a convex set with $m(V)>0$. Then, $$m(V\cap B(x,r))>0\text{ for any }x\in \overline{V}\text{ and }r>0.$$ 
\end{lemma}
\begin{proof}
From the proof of Lemma \ref{pL8}, $m(\partial V)=0$ so $$m(V^\circ)=m(V)-m(\partial V\cap V)=m(V)>0.$$ Thus $V^\circ$ is uncountable and $V^\circ=\cup_{u\in V^\circ}B(u,r_u)$ for some $r_u>0$. For $x\in\overline{V}$ and $r>0$, by the convexity of $\overline{V}$, there exists $y\in V^\circ$ with $\|y-x\|<\frac{r}{2},$ so by the triangle inequality, $B(y,\frac{r}{2})\subset B(x,r)$. Hence $$V\cap B(x,r)\supset B(y,r_y)\cap B(y,\frac{r}{2})=B(y,r_y\wedge\frac{r}{2}),\text{ so } m(V\cap B(x,r))\geq m(B(y,r_y\wedge\frac{r}{2}))>0.$$ This completes the proof.
\end{proof}

\begin{proof}[Proof of Theorem \ref{pT1}]
From Lemma \ref{pL4}, $\{\mathscr{G}(0)\hbox{ is connected}\}\subset\mathscr{C}$ for $1\leq n\leq 4$ so $$P(\mathscr{C})\geq P(\mathscr{G}(0)\hbox{ is connected})\hbox{ for }1\leq n\leq 4.$$\\
Observe that $\mathscr{C}\supset\{\mathscr{G}(0)\text{ is }\epsilon\text{-trivial}\}\supset\{x_i(0)\in B(x_1(0),\epsilon/2)\text{ for all }i\in [n]\}$ so 
\begin{align*}
  P(\mathscr{C})&\geq P(\mathscr{G}(0)\text{ is }\epsilon\text{-trivial})\\
  &\geq P(x_i(0)\in B(x_1(0),\epsilon/2)\text{ for all }i\in[n])\\
  &=\int_{\mathbf{R^d}}\int_{B(x_1,\epsilon/2)}...\int_{B(x_1,\epsilon/2)}\prod_{i=1}^nf(x_i)dm(x_n)...dm(x_1)\\
  &:=\int_{\mathbf{R^d}}(\int_{B(x_1,\epsilon/2)})^{n-1}\prod_{i=1}^nf(x_i)dm(x_n)...dm(x_1)\\
  &=\int_{\mathbf{R^d}}f(x_1)\left(\int_{B(x_1,\epsilon/2)}f(x)dm(x)\right)^{n-1}dm(x_1).  
\end{align*}
Observe that $f>0$ on the convex set $S$ and $m(B(x_1,\epsilon/2)\cap S)>0$ for all $x_1\in S$ from Lemma \ref{pL9}. Hence via Lemma \ref{pL6}, 
\begin{align*}
  P(\mathscr{C})&\geq \int_{\mathbf{R^d}}f(x_1)\left(\int_{B(x_1,\epsilon/2)}f(x)dm(x)\right)^{n-1}dm(x_1)\\
  &=\int_{S}f(x_1)\left(\int_{B(x_1,\epsilon/2)\cap S}f(x)dm(x)\right)^{n-1}dm(x_1)>0.
\end{align*}
\end{proof}

\begin{proof}[Proof of Corollary \ref{pCo1}]
From theorem \ref{pT1}, 
\begin{align*}
    P(\mathscr{C})&\geq P(x_i(0)\in B(x_1(0),\epsilon/2)\text{ for all }i\in[n])\\
    &\geq\int_{[\epsilon/2,1-\epsilon/2]^d}dm(x_1)\left(\int_{B(x_1,\epsilon/2)}1 dm(x)\right)^{n-1}\\
    &=\int_{[\epsilon/2,1-\epsilon/2]^d}m(B(x_1,\frac{\epsilon}{2}))^{n-1}dm(x_1)\\
    &=\left((\frac{\epsilon}{2})^d m(B(0,1))\right)^{n-1}(1-\epsilon)^d
    =\left((\frac{\epsilon}{2})^d \frac{\pi^{\frac{d}{2}}}{\Gamma(\frac{d}{2}+1)}\right)^{n-1}(1-\epsilon)^d.
\end{align*}
\end{proof}

\section{One dimensional probability of consensus}
In this section, we focus on the one dimensional HK model. Apart from higher dimensions, opinions in one dimension are ordered by $\leq$. We demonstrate that opinions are order-preserving and profiles are disconnected-preserving. Hence $P(\mathscr{C})=P(\mathscr{G}(0)\text{ is connected})$ for $1\leq n\leq 4$ and in general $P(\mathscr{G}(0)\text{ is connected})$ is an upper bound for the probability of consensus. Furthermore, we demonstrate the probability of consensus on $[0,1].$ 

\begin{lemma}[order-preserving]\label{pL10} For d=1, if $x_i(t)\leq x_j(t)$ then $x_i(t+1)\leq x_j(t+1).$
\end{lemma}
\begin{proof}
Let $x=x(t),x'=x(t+1),$ and $N_i=N_i(t)$ for all $i\in[n].$ From Lemma \ref{pL1},
\begin{align*}
    &x_j'-x_i'=(\frac{1}{|N_j|}-\frac{1}{|N_i|})\sum_{k\in N_i\cap N_j}x_k+\frac{1}{|N_j|}\sum_{k\in N_j-N_i}x_k-\frac{1}{|N_i|}\sum_{k\in N_i-N_j}x_k\\
    &=\left\{\begin{array}{lr}
         \displaystyle \sum_{k,p\in N_j-N_i,q\in N_i}a_k(x_p-x_q)&\hbox{if}\ |N_j|\geq |N_i|  \\
         \displaystyle \sum_{k,p\in N_j,q\in N_i-N_j}a_k(x_p-x_q)&\hbox{else,}
    \end{array}\right.
\end{align*}
 where $a_k\geq0$ for all $k.$ We claim that
 \begin{enumerate}
     \item If $a\in N_i$ and $b\in N_j-N_i$, then $x_a<x_b$.\label{v1}
     \item If $a\in N_i-N_j$ and $b\in N_j$, then $x_a<x_b$.\label{v2}
 \end{enumerate}
 \begin{proof}[Proof of Claim \ref{v1}]
 Assume by contradiction that there exist $a\in N_i$ and $b\in N_j-N_i$ such that $x_a\geq x_b$. Then, $x_j-\epsilon\leq x_b< x_i-\epsilon$, a contradiction.
 \end{proof} 
 \begin{proof}[Proof of Claim \ref{v2}]
  Assume by contradiction that there exist $a\in N_i-N_j$ and $b\in N_j$ such that $x_a\geq x_b$. Then, $x_i+\epsilon\geq x_a> x_j+\epsilon$, a contradiction.
 \end{proof}
 Either way, $x_j'-x_i'\geq0.$ This completes the proof.
\end{proof}

\begin{lemma}[disconnected-preserving]\label{pL11} For $d=1,$ if $\mathscr{G}(t)$ is disconnected then $\mathscr{G}(t+1)$ is disconnected.
\end{lemma}
\begin{proof}
Assume $x_1(t)\leq x_2(t)\leq...\leq x_n(t).$ Since $\mathscr{G}(t)$ is disconnected, $$x_{i+1}(t)-x_i(t)>\epsilon\text{ for some }i\in[n-1].$$ Since vertices $i$ and $i+1$ have respectively no neighbors on its right and left at time $t$, $$x_i(t+1)\leq x_i(t)\text{ and }x_{i+1}(t)\leq x_{i+1}(t+1).$$ Hence $x_{i+1}(t+1)-x_i(t+1)>\epsilon.$ From Lemma \ref{pL10}, $$x_1(t+1)\leq x_2(t+1)\leq...\leq x_n(t+1).$$ Thus $\mathscr{G}(t+1)$ is disconnected.
\end{proof}
 Next, we consider several circumstances under which a profile is connected at the next time step. Let $M\subset\mathbf{R}$ be a finite nonempty set and $\overline{M}=\frac{\sum_{x\in M}x}{|M|}$ be the average on $M.$ It is clear that
\begin{equation}\label{av}
 \overline{a+M}>\overline{M} \iff a>\overline{M}.   
\end{equation}
\begin{lemma}
 For $4\leq n\leq 7$, if $\mathscr{G}(t)$ satisfies $(\ast)$, then so does $\mathscr{G}(t+1).$
\end{lemma}
\begin{proof}
Let $x=x(t)$, $x'=x(t+1),\ \epsilon_1=x_{(m+2)}-x_{(1)},\ \epsilon_2=x_{(n)}-x_{(k)},\ y_1=x_{(1)},\ y_2=x_{(m+2)},\ y_3=x_{(k)},\ y_4=x_{(n)}$. Since $x_i's$ are order-preserving, need only consider $y_3'-y_2',\ y_2'-y_1',\text{ and }y_4'-y_3'.$ By \eqref{av}, 
\begin{align*}
    &\max_{(x_i)_{i=1}^n-\{y_2,y_3\}}(y_3'-y_2')=\frac{y_2+y_3+(y_3+\epsilon_2)+[sy_2+(k-m-3-s)y_3]+m(y_3+\epsilon_2)}{k}\\
    &-\frac{(y_2-\epsilon_1)+y_2+y_3+m(y_2-\epsilon_1)+[sy_2+(k-m-3-s)y_3]}{k}\text{ for some }0\leq s\leq k-m-3,\\
    &=\frac{(y_3-y_2)+\epsilon_2+\epsilon_1+m(y_3-y_2)+m(\epsilon_2+\epsilon_1)}{k}=\frac{(m+1)(y_3-y_2)+(m+1)(\epsilon_1+\epsilon_2)}{k}\\
    &\leq\frac{2(m+1)}{k}\epsilon<\epsilon\text{ so }((m+2),(k))\in\mathscr{E}(t+1).
\end{align*}
\begin{align*}
    &\max_{(x_i)_{i=1}^n-\{y_1,y_2\}}(y_2'-y_1')=\frac{y_1+y_2+(n-m-2)(y_2+\epsilon)+[sy_1+(m-s)y_2]}{n}\\
    &-\frac{y_1+y_2+[sy_1+(m-s)y_2]}{m+2}\text{ for some }0\leq s\leq m.
\end{align*}
Since $$\partial_{s}\max_{(x_i)_{i=1}^n-\{y_1,y_2\}}(y_2'-y_1')=\frac{y_1-y_2}{n}-\frac{y_1-y_2}{m+2}=(y_1-y_2)(\frac{1}{n}-\frac{1}{m+2})\geq0,\text{ set }s=m,$$
\begin{align*}
    \max_{(x_i)_{i=1}^n-\{y_1,y_2\}}(y_2'-y_1')&=\frac{(m+2)[(m+1)y_1+(n-m-1)y_2+(n-m-2)\epsilon]-n[(m+1)y_1+y_2]}{n(m+2)}\\
    &=\frac{(m+1)(n-m-2)(y_2-y_1)+(m+2)(n-m-2)\epsilon}{n(m+2)}.
\end{align*}
By symmetry, $$\max_{(x_i)_{i=1}^n-\{y_3,y_4\}}(y_4'-y_3')=\frac{(m+1)(n-m-2)(y_4-y_3)+(m+2)(n-m-2)\epsilon}{n(m+2)}.$$ Hence
\begin{align*}
    &\max_{(x_i)_{i=1}^n-\{y_1,y_2\}}(y_2'-y_1')+\max_{(x_i)_{i=1}^n-\{y_3,y_4\}}(y_4'-y_3')\\
    &=\frac{(m+1)(n-m-2)(y_2-y_1+y_4-y_3)+(m+2)(n-m-2)\epsilon}{n(m+2)}\\
    &\leq \frac{(n-m-2)(2m+3)}{n(m+2)}\epsilon\leq\epsilon\text{ for }4\leq n\leq 7. 
\end{align*}
So $y_4'-y_3'+y_2'-y_1'\leq\epsilon$ for $4\leq n\leq7$. This completes the proof.
\end{proof}
 Observe that a profile is connected-preserving if it satisfies $(\ast)$. It is clear that an $\epsilon$-trivial profile satisfies $(\ast)$ and there exists an $\epsilon$-nontrivial profile satisfies $(\ast).$ Thus $\{\mathscr{G}(0)\text{ satisfies }(\ast)\}\supsetneq\{\mathscr{G}(0)\text{ is }\epsilon\text{-trivial}\}$.
\begin{lemma}\label{pL13}
For any $0\leq i\leq m$ and $n\geq 4$, assume that $$\max\left((x_{(n)}-x_{(n-i-1)})(t),\ (x_{(n-i-1)}-x_{(i+2)})(t),\ (x_{(i+2)}-x_{(1)})(t)\right)\leq\frac{\epsilon}{2}.$$ Then, $$\max\left((x_{(n)}-x_{(n-i-1)})(t+1),\ (x_{(n-i-1)}-x_{(i+2)})(t+1),\ (x_{(i+2)}-x_{(1)})(t+1)\right)<\frac{\epsilon}{2}.$$ 
\end{lemma}
\begin{proof}
Let $x=x(t),\ x'=x(t+1),\ y_1=x_{(1)},\ y_2=x_{(i+2)},\ y_3=x_{(n-i-1)},\ y_4=x_{(n)}.$ By the assumption, the neighborhood of $(i+2)$ is the same as that of $(n-i-1),$ so $y_3'-y_2'=0.$ Via \eqref{av}, 
\begin{align*}
    &\max_{(x_i)_{i=1}^n-\{y_1,y_2\}}(y_2'-y_1')\\
    &=\frac{y_1+y_2+[s_1y_1+(i-s_1)y_2]+[s_2y_2+(n-2i-3-s_2)(y_2+\frac{\epsilon}{2})]+(i+1)(y_2+\epsilon)}{n}\\
    &-\frac{y_1+y_2+[s_1y_1+(i-s_1)y_2]+[s_2y_2+(n-2i-3-s_2)(y_2+\frac{\epsilon}{2})]}{n-i-1}
\end{align*} for some $0\leq s_1\leq i$ and $0\leq s_2\leq n-2i-3$. Since $3m+4\leq n\leq3m+6$, $0\leq i\leq m$, $\partial_{s_1}\max_{(x_i)_{i=1}^n-\{y_1,y_2\}}(y_2'-y_1')=(y_1-y_2)(\frac{1}{n}-\frac{1}{n-i-1})\geq0$, and $\partial_{s_2}\max_{(x_i)_{i=1}^n-\{y_1,y_2\}}(y_2'-y_1')=[y_2-(y_2+\frac{\epsilon}{2})](\frac{1}{n}-\frac{1}{n-i-1})\geq0$, set $s_1=i$ and $s_2=n-2i-3,$
\begin{align*}
    &\max_{(x_i)_{i=1}^n-\{y_1,y_2\}}(y_2'-y_1')\\
    &=\frac{(i+1)y_1+(n-i-1)y_2+(i+1)\epsilon}{n}-\frac{(i+1)y_1+(n-2i-2)y_2}{n-i-1}\\
    &=\frac{(n-i-1)[(i+1)y_1+(n-i-1)y_2+(i+1)\epsilon]-n[(i+1)y_1+(n-2i-2)y_2]}{n(n-i-1)}\\
    &=\frac{(i+1)^2(y_2-y_1)+(n-i-1)(i+1)\epsilon}{n(n-i-1)}\leq\frac{\frac{(i+1)^2}{2}+(n-i-1)(i+1)}{n(n-i-1)}\epsilon=\frac{1}{2}\frac{(i+1)(2n-i-1)}{n(n-i-1)}\epsilon\\
    &\leq\frac{1}{2}\frac{(m+1)(2n-1)}{n(n-m-1)}\epsilon\leq\frac{1}{2}\frac{(m+1)(6m+11)}{(3m+4)(2m+3)}\epsilon<\frac{\epsilon}{2}.
\end{align*} By symmetry, $$\max_{(x_i)_{i=1}^n-\{y_3,y_4\}}(y_4'-y_3')=\frac{(i+1)^2(y_4-y_3)+(n-i-1)(i+1)\epsilon}{n(n-i-1)}<\frac{\epsilon}{2}.$$ This completes the proof. 
\end{proof}
 Observe that an $\epsilon$-trivial profile may not satisfies the assumption of Lemma \ref{pL13}. Consider $n>1,\ x_1(t)=0,\ x_i(t)=\epsilon$ for $i>1$ then $\mathscr{G}(t)$ is $\epsilon$-trivial but does not satisfy the assumption of Lemma \ref{pL13}. Observe that $\mathscr{G}(s)$ satisfies $(\ast\ast)$ for all $s\geq t.$

\begin{proof}[Proof of Theorem \ref{pT2}]From Lemmas \ref{pL4} and \ref{pL11}, $\mathscr{G}(t)$ is connected-preserving and disconnected-preserving for $1\leq n\leq 4$ so $\mathscr{C}=\{\mathscr{G}(0)\text{ is connected}\}.$ Thus $$P(\mathscr{C})=P(\mathscr{G}(0)\text{ is connected})\text{ for }1\leq n\leq 4.$$ Since $\mathscr{G}(t)$ is disconnected-preserving for all $n\geq1$ and $\mathscr{G}(t)$ is connected-preserving if it satisfies $(\ast\ast)$ for all $n\geq4,$
$$\{\mathscr{G}(0)\text{ is connected}\}\supset\mathscr{C}\supset\{\mathscr{G}(0)\text{ is }\epsilon\text{-trivial}\}\cup\{\mathscr{G}(0)\text{ satisfies }(\ast\ast)\}.$$ Hence $$P(\mathscr{G}(0)\text{ is connected})\geq P(\mathscr{C})\geq P(\mathscr{G}(0)\text{ is }\epsilon\text{-trivial or satisfies }(\ast\ast)).$$

\end{proof}

\begin{proof}[Proof of Corollary \ref{pCo2}]
For $n=2,$
$$P(\mathscr{C})=2!\int_{[0,1]}dx_1\int_{[x_1,x_1+\epsilon]\cap [0,1]}dx_2.$$
\begin{align*}
    &\int_{[0,1]}dx_1\int_{[x_1,x_1+\epsilon]\cap [0,1]}dx_2=\int_{[0,1-\epsilon]+[1-\epsilon,1]}\left[(x_1+\epsilon)\wedge1-x_1\right]dx_1\\
    &=\int_{[0,1-\epsilon]}\epsilon dx_1+\int_{[1-\epsilon,1]}1-x_1 dx_1\\
    &=\epsilon(1-\epsilon)-\left[\frac{(1-x_1)^2}{2}\right]_{1-\epsilon}^1=\epsilon(1-\epsilon)+\frac{1}{2}\epsilon^2=\epsilon(1-\frac{\epsilon}{2}).
\end{align*}
Thus $$P(\mathscr{C})=\epsilon(2-\epsilon).$$
For $n=3,$
\begin{align*}
    &P(\mathscr{C})=3!\int_{[0,1]}dx_1\int_{[x_1,x_1+\epsilon]\cap [0,1]}dx_2\int_{[x_2,x_2+\epsilon]\cap [0,1]}dx_3.
\end{align*}
(i) $\epsilon\in[\frac{1}{2},1)$\\
\begin{align*}
    &\int_{[0,1]}dx_1\int_{[x_1,x_1+\epsilon]\cap [0,1]}dx_2\int_{[x_2,x_2+\epsilon]\cap [0,1]}dx_3\\
    &=\int_{[0,1]}dx_1\int_{[x_1,(x_1+\epsilon)\wedge 1]}dx_2\int_{[x_2,(x_2+\epsilon)\wedge 1]}dx_3\\
    &=\int_{[0,1]}dx_1\int_{[x_1,(x_1+\epsilon)\wedge 1]}(x_2+\epsilon)\wedge 1-x_2 dx_2\\
    &=\int_{[0,1]}dx_1\left(\int_{[x_1,(x_1+\epsilon)\wedge 1]\cap [0,1-\epsilon]}\epsilon dx_2+\int_{[x_1,(x_1+\epsilon)\wedge 1]\cap[1-\epsilon,1]}1-x_2 dx_2\right)\\
    &=\int_{[0,1-\epsilon]}dx_1\left(\int_{[x_1,1-\epsilon]}\epsilon dx_2+\int_{[1-\epsilon,x_1+\epsilon]}1-x_2 dx_2\right)+\int_{[1-\epsilon,1]}dx_1\int_{[x_1,1]}1-x_2 dx_2\\
    &=\int_{[0,1-\epsilon]}\epsilon(1-\epsilon-x_1)-\left[\frac{(1-x_2)^2}{2}\right]_{x_2=1-\epsilon}^{x_1+\epsilon} dx_1+\int_{[1-\epsilon,1]}-\left[\frac{(1-x_2)^2}{2}\right]_{x_2=x_1}^1 dx_1\\
    &\frac{-\epsilon(1-\epsilon-x_1)^2}{2}|_0^{1-\epsilon}+\frac{1}{2}\int_{[0,1-\epsilon]}\epsilon^2-(1-\epsilon-x_1)^2 dx_1+\frac{1}{2}\int_{[1-\epsilon,1]}(1-x_1)^2 dx_1\\
    &=\frac{1}{2}\left\{\epsilon(1-\epsilon)^2+\epsilon^2(1-\epsilon)+\left[\frac{(1-\epsilon-x_1)^3}{3}\right]_0^{1-\epsilon}-\left[\frac{(1-x_1)^3}{3}\right]_{1-\epsilon}^1\right\}\\
    &=\frac{1}{2}\left\{\epsilon(1-\epsilon)-\frac{(1-\epsilon)^3}{3}+\frac{\epsilon^3}{3}\right\}\text{ so }\\
    &P(\mathscr{C})=3\epsilon(1-\epsilon)-(1-\epsilon)^3+\epsilon^3=\epsilon^3+(1-\epsilon)^3+3\epsilon(1-\epsilon)-2(1-\epsilon)^3\\
    &=\epsilon^2-\epsilon(1-\epsilon)+(1-\epsilon)^2+3\epsilon(1-\epsilon)-2(1-\epsilon)^2\\
    &=\epsilon^2+2\epsilon(1-\epsilon)+(1-\epsilon)^2-2(1-\epsilon)^3=1-2(1-\epsilon)^3.
\end{align*}
(ii) $\epsilon\in(0,\frac{1}{2})$
\begin{align*}
   &\int_{[0,1]}dx_1\int_{[x_1,x_1+\epsilon]\cap [0,1]}dx_2\int_{[x_2,x_2+\epsilon]\cap [0,1]}dx_3\\
   &=\int_{[0,1]}dx_1\left(\int_{[x_1,(x_1+\epsilon)\wedge 1]\cap [0,1-\epsilon]}\epsilon dx_2+\int_{[x_1,(x_1+\epsilon)\wedge 1]\cap[1-\epsilon,1]}1-x_2 dx_2\right)\\
   &=\int_{[0,1-2\epsilon]}dx_1\left(\int_{[x_1,x_1+\epsilon]\cap[0,1-\epsilon]}\epsilon dx_2+\int_{[x_1,x_1+\epsilon]\cap[1-\epsilon,1]}1-x_2 dx_2\right)\\
   &+\int_{[1-2\epsilon,1-\epsilon]}dx_1\left(\int_{[x_1,x_1+\epsilon]\cap[0,1-\epsilon]}\epsilon dx_2+\int_{[x_1,x_1+\epsilon]\cap[1-\epsilon,1]}1-x_2 dx_2\right)\\
   &+\int_{[1-\epsilon,1]}dx_1\left(\int_{[x_1,1]\cap[0,1-\epsilon]}\epsilon dx_2+\int_{[x_1,1]\cap[1-\epsilon,1]}1-x_2 dx_2\right)\\
   &=\int_{[0,1-2\epsilon]}\epsilon^2 dx_1+\left(\int_{[1-2\epsilon,1-\epsilon]}\epsilon(1-\epsilon-x_1)-\left[\frac{(1-x_2)^2}{2}\right]_{x_2=1-\epsilon}^{x_1+\epsilon}dx_1\right)\\
   &+\int_{[1-\epsilon,1]}-\left[\frac{(1-x_2)^2}{2}\right]_{x_2=x_1}^1 dx_1\\
   &=\epsilon^2(1-2\epsilon)+\int_{[1-2\epsilon,1-\epsilon]}\epsilon(1-\epsilon-x_1)\\
   &+\frac{1}{2}\left[\epsilon^2-(1-\epsilon-x_1)^2\right]dx_1+\frac{1}{2}\int_{[1-\epsilon,1]}(1-x_1)^2 dx_1\\
   &=\epsilon^2(1-2\epsilon)-\left[\frac{\epsilon(1-\epsilon-x_1)^2}{2}\right]_{1-2\epsilon}^{1-\epsilon}+\frac{1}{2}\left\{\epsilon^3+\left[\frac{(1-\epsilon-x_1)^3}{3}\right]_{x_1=1-2\epsilon}^{1-\epsilon}\right\}\\
   &-\frac{1}{2}\left[\frac{(1-x_1)^3}{3}\right]_{1-\epsilon}^1\\
   &=\epsilon^2(1-2\epsilon)+\frac{1}{2}(\epsilon^3+\epsilon^3-\frac{\epsilon^3}{3}+\frac{\epsilon^3}{3})=\epsilon^2(1-2\epsilon)+\epsilon^3=\epsilon^2(1-\epsilon)\text{ so}\\
   &P(\mathscr{C})=6\epsilon^2(1-\epsilon).
\end{align*}
Thus $$P(\mathscr{C})=\left\{\begin{array}{cc}
     6\epsilon^2(1-\epsilon)&\epsilon\in(0,\frac{1}{2})  \\
     1-2(1-\epsilon)^3& \epsilon\in[\frac{1}{2},1)
\end{array}\right.$$
For $n=4,$
$$P(\mathscr{C})=4!\int_{[0,1]}dx_1\int_{[x_1,(x_1+\epsilon)\wedge 1]}dx_2\int_{[x_2,(x_2+\epsilon)\wedge 1]}dx_3\int_{[x_3,(x_3+\epsilon)\wedge 1]}dx_4.$$
(i) $\epsilon\in[\frac{1}{2},1)$
\begin{align*}
  & \int_{[0,1]}dx_1\int_{[x_1,(x_1+\epsilon)\wedge 1]}dx_2\int_{[x_2,(x_2+\epsilon)\wedge 1]}dx_3\int_{[x_3,(x_3+\epsilon)\wedge 1]}dx_4\\
  &=\int_{[0,1]}dx_1\int_{[x_1,(x_1+\epsilon)\wedge 1]}dx_2\int_{[x_2,(x_2+\epsilon)\wedge 1]}(x_3+\epsilon)\wedge 1-x_3 dx_3\\
  &=\int_{[0,1]}dx_1\int_{[x_1,(x_1+\epsilon)\wedge 1]}dx_2\left(\int_{[x_2,(x_2+\epsilon)\wedge 1]\cap[0,1-\epsilon]}\epsilon dx_3+\int_{[x_2,(x_2+\epsilon)\wedge 1]\cap[1-\epsilon,1]}1-x_3 dx_3\right)\\
  &=\int_{[0,1]}dx_1\left[\int_{[x_1,(x_1+\epsilon)\wedge 1]\cap[0,1-\epsilon]}\left(\int_{[x_2,1-\epsilon]}\epsilon dx_3+\int_{[1-\epsilon,x_2+\epsilon]}1-x_3 dx_3\right)\right.\\
  &\left.+\int_{[x_1,(x_1+\epsilon)\wedge 1]\cap[1-\epsilon,1]}dx_2\left(\int_{[x_2,1]}1-x_3 dx_3\right)\right]\\
  &=\int_{[0,1]}dx_1\left(\int_{[x_1,(x_1+\epsilon)\wedge 1]\cap[0,1-\epsilon]}\epsilon(1-\epsilon-x_2)-\left[\frac{(1-x_3)^2}{2}\right]_{1-\epsilon}^{x_2+\epsilon}dx_2\right.\\
  &\left.+\int_{[x_1,(x_1+\epsilon)\wedge 1]\cap[1-\epsilon,1]}-\left[\frac{(1-x_3)^2}{2}\right]_{x_2}^1 dx_2\right)\\
  &=\int_{[0,1-\epsilon]}dx_1\left(\int_{[x_1,1-\epsilon]}\epsilon(1-\epsilon-x_2)+\frac{1}{2}[\epsilon^2-(1-\epsilon-x_2)^2]dx_2\right.\\
  &\left.+\frac{1}{2}\int_{[1-\epsilon,x_1+\epsilon]}(1-x_2)^2dx_2\right)
  +\int_{[1-\epsilon,1]}dx_1\left(\frac{1}{2}\int_{[x_1,1]}(1-x_2)^2dx_2\right)\\
  &=\int_{[0,1-\epsilon]}-\left[\frac{\epsilon(1-\epsilon-x_2)^2}{2}\right]_{x_1}^{1-\epsilon}+\frac{1}{2}\left(\epsilon^2(1-\epsilon-x_1)+\left[\frac{(1-\epsilon-x_2)^3}{3}\right]_{x_1}^{1-\epsilon}\right)\\
  &-\frac{1}{2}\left[\frac{(1-x_2)^3}{3}\right]_{1-\epsilon}^{x_1+\epsilon}dx_1
  +\int_{[1-\epsilon,1]}-\frac{1}{2}\left[\frac{(1-x_2)^3}{3}\right]_{x_1}^1dx_1\\
  &=\frac{1}{2}\left(\int_{[0,1-\epsilon]}\epsilon(1-\epsilon-x_1)^2+\epsilon^2(1-\epsilon-x_1)-\frac{(1-\epsilon-x_1)^3}{3}\right.\\
  &\left.+\frac{[\epsilon^3-(1-\epsilon-x_1)^3]}{3}dx_1
  +\int_{[1-\epsilon,1]}\frac{(1-x_1)^3}{3}dx_1\right)\\
  &=\frac{1}{2}\left\{-\left[\frac{\epsilon(1-\epsilon-x_1)^3}{3}\right]_0^{1-\epsilon}-\left[\frac{\epsilon^2(1-\epsilon-x_1)^2}{2}\right]_0^{1-\epsilon}\right.\\
  &\left.+\frac{1}{3}\left(\epsilon^3(1-\epsilon)+\left[\frac{2(1-\epsilon-x_1)^4}{4}\right]_0^{1-\epsilon}-\frac{1}{3}\left[\frac{(1-x_1)^4}{4}\right]_{1-\epsilon}^1\right)\right\}\\
  &=\frac{1}{2}\left\{\frac{\epsilon(1-\epsilon)^3}{3}+\frac{\epsilon^2(1-\epsilon)^2}{2}+\frac{1}{3}\left[\epsilon^3(1-\epsilon)-\frac{2(1-\epsilon)^4}{4}\right]+\frac{\epsilon^4}{12}\right\}\\
  &=\frac{1}{2}\left(\frac{\epsilon(1-\epsilon)^3}{3}+\frac{\epsilon^2(1-\epsilon)^2}{2}+\frac{\epsilon^3(1-\epsilon)}{3}-\frac{(1-\epsilon)^4}{6}+\frac{\epsilon^4}{12}\right)\\
  &=\frac{1}{24}\left(\epsilon^4+4\epsilon^3(1-\epsilon)+6\epsilon^2(1-\epsilon)^2+4\epsilon(1-\epsilon)^3-2(1-\epsilon)^4\right)\text{ so }\\
  &P(\mathscr{C})=\epsilon^4+4\epsilon^3(1-\epsilon)+6\epsilon^2(1-\epsilon)^2+4\epsilon(1-\epsilon)^3-2(1-\epsilon)^4.
\end{align*}
(ii) $\epsilon\in[\frac{1}{3},\frac{1}{2})$
\begin{align*}
    &\int_{[0,1]}dx_1\int_{[x_1,(x_1+\epsilon)\wedge 1]}dx_2\int_{[x_2,(x_2+\epsilon)\wedge 1]}dx_3\int_{[x_3,(x_3+\epsilon)\wedge 1]}dx_4\\
    &=\int_{[0,1]}dx_1\int_{[x_1,(x_1+\epsilon)\wedge 1]}dx_2\left(\int_{[x_2,(x_2+\epsilon)\wedge 1]\cap[0,1-\epsilon]}\epsilon dx_3+\int_{[x_2,(x_2+\epsilon)\wedge 1]}1-x_3 dx_3\right)\\
    &=\int_{[0,1]}dx_1\left[\int_{[x_1,(x_1+\epsilon)\wedge 1]\cap[0,1-2\epsilon]}dx_2\left(\int_{[x_2,x_2+\epsilon]}\epsilon dx_3\right)\right.\\
    &+\int_{[x_1,(x_1+\epsilon)\wedge 1]\cap[1-2\epsilon,1-\epsilon]}dx_2\left(\int_{[x_2,1-\epsilon]}\epsilon dx_3+\int_{[1-\epsilon,x_2+\epsilon]}1-x_3 dx_3\right)\\
    &\left.+\int_{[x_1,(x_1+\epsilon)\wedge 1]\cap[1-\epsilon,1]}dx_2\int_{[x_2,1]}1-x_3 dx_3\right]\\
    &=\int_{[0,1]}dx_1\left(\int_{[x_1,(x_1+\epsilon)\wedge 1]\cap[0,1-2\epsilon]}\epsilon^2 dx_2\right.\\
    &+\int_{[x_1,(x_1+\epsilon)\wedge 1]\cap[1-2\epsilon,1-\epsilon]}\epsilon(1-\epsilon-x_2)-\left[\frac{(1-x_3)^2}{2}\right]_{1-\epsilon}^{x_2+\epsilon}dx_2\\
    &\left.+\int_{[x_1,(x_1+\epsilon)\wedge 1]\cap[1-\epsilon,1]}-\left[\frac{(1-x_3)^2}{2}\right]_{x_2}^1dx_2\right)\\
    &=\int_{[0,1-2\epsilon]}dx_1\left(\int_{[x_1,1-2\epsilon]}\epsilon^2 dx_2+\int_{[1-2\epsilon,x_1+\epsilon]}\epsilon(1-\epsilon-x_2)\right.\\
    &\left.+\frac{1}{2}[\epsilon^2-(1-\epsilon-x_2)^2]dx_2\right)\\
    &+\int_{[1-2\epsilon,1-\epsilon]}dx_1\left(\int_{[x_1,1-\epsilon]}\epsilon(1-\epsilon-x_2)+\frac{1}{2}[\epsilon^2-(1-\epsilon-x_2)^2]dx_2\right.\\
    &\left.+\frac{1}{2}\int_{[1-\epsilon,x_1+\epsilon]}(1-x_2)^2 dx_2\right)
    +\frac{1}{2}\int_{[1-\epsilon,1]}dx_1\int_{[x_1,1]}(1-x_2)^2dx_2\\
    &=\int_{[0,1-2\epsilon]}\epsilon^2(1-2\epsilon-x_1)-\left[\frac{\epsilon(1-\epsilon-x_2)^2}{2}\right]_{1-2\epsilon}^{x_1+\epsilon}\\
    &+\frac{1}{2}\left(\epsilon^2(x_1+3\epsilon-1)+\left[\frac{(1-\epsilon-x_2)^3}{3}\right]_{1-2\epsilon}^{x_1+\epsilon}\right)dx_1\\
    &+\int_{[1-2\epsilon,1-\epsilon]}-\left[\frac{\epsilon(1-\epsilon-x_2)^2}{2}\right]_{x_1}^{1-\epsilon}+\frac{1}{2}\left(\epsilon^2(1-\epsilon-x_1)+\left[\frac{(1-\epsilon-x_2)^3}{3}\right]_{x_1}^{1-\epsilon}\right)\\
    &-\frac{1}{2}\left[\frac{(1-x_2)^3}{3}\right]_{1-\epsilon}^{x_1+\epsilon}dx_1-\frac{1}{2}\int_{[1-\epsilon,1]}\left[\frac{(1-x_2)^3}{3}\right]_{x_1}^1 dx_1\\
    &=\int_{[0,1-2\epsilon]}\epsilon^2(1-2\epsilon-x_1)+\frac{1}{2}[\epsilon^3-\epsilon(1-2\epsilon-x_1)^2]\\
    &+\frac{1}{2}\left(\epsilon^2(x_1+3\epsilon-1)+\frac{1}{3}[(1-2\epsilon-x_1)^3-\epsilon^3]\right)dx_1+\int_{[1-2\epsilon,1-\epsilon]}\frac{\epsilon(1-\epsilon-x_1)^2}{2}\\
    &+\frac{1}{2}\left(\epsilon^2(1-\epsilon-x_1)-\frac{(1-\epsilon-x_1)^3}{3}\right)
    +\frac{1}{6}[\epsilon^3-(1-\epsilon-x_1)^3]dx_1\\
    &+\frac{1}{2}\int_{[1-\epsilon,1]}\frac{1}{3}(1-x_1)^3 dx_1\\
    &=-\left[\frac{\epsilon^2(1-2\epsilon-x_1)^2}{2}\right]_0^{1-2\epsilon}+\frac{1}{2}\left(\epsilon^3(1-2\epsilon)+\left[\frac{\epsilon(1-2\epsilon-x_1)^3}{3}\right]_0^{1-2\epsilon}\right)\\
    &+\frac{1}{2}\left[\left[\frac{\epsilon^2(x_1+3\epsilon-1)^2}{2}\right]_0^{1-2\epsilon}+\frac{1}{3}\left(-\left[\frac{(1-2\epsilon-x_1)^4}{4}\right]_0^{1-2\epsilon}-\epsilon^3(1-2\epsilon)\right)\right]\\
    &+\frac{1}{2}\left\{-\left[\frac{\epsilon(1-\epsilon-x_1)^3}{3}\right]_{1-2\epsilon}^{1-\epsilon}-\left[\frac{\epsilon^2(1-\epsilon-x_1)^2}{2}\right]_{1-2\epsilon}^{1-\epsilon}+\left[\frac{(1-\epsilon-x_1)^4}{12}\right]_{1-2\epsilon}^{1-\epsilon}\right.\\
    &\left.+\frac{1}{3}\left(\epsilon^4+\left[\frac{(1-\epsilon-x_1)^4}{4}\right]_{1-2\epsilon}^{1-\epsilon}\right)\right\}-\frac{1}{6}\left[\frac{(1-x_1)^4}{4}\right]_{1-\epsilon}^1\\
    &=\frac{\epsilon^2(1-2\epsilon)^2}{2}+\frac{1}{2}[\epsilon^3(1-2\epsilon)-\frac{\epsilon(1-2\epsilon)^3}{3}]+\frac{1}{4}[\epsilon^4-\epsilon^2(3\epsilon-1)^2]+\frac{1}{24}(1-2\epsilon)^4\\
    &-\frac{1}{6}\epsilon^3(1-2\epsilon)
    +\frac{1}{6}\epsilon^4+\frac{1}{4}\epsilon^4-\frac{1}{24}\epsilon^4+\frac{1}{6}\epsilon^4-\frac{1}{24}\epsilon^4+\frac{1}{24}\epsilon^4\\
    &=\frac{19}{24}\epsilon^4-\frac{1}{6}\epsilon^3(1-2\epsilon)+\frac{1}{24}(1-2\epsilon)^4-\frac{1}{4}\epsilon^2(3\epsilon-1)^2-\frac{1}{6}\epsilon(1-2\epsilon)^3\\
    &+\frac{1}{2}\epsilon^3(1-2\epsilon)+\frac{1}{2}\epsilon^2(1-2\epsilon)^2\text{ so}\\
    P(\mathscr{C})&=19\epsilon^4-4\epsilon^3(1-2\epsilon)+(1-2\epsilon)^4-6\epsilon^2(3\epsilon-1)^2-4\epsilon(1-2\epsilon)^3\\
    &+12\epsilon^3(1-2\epsilon)+12\epsilon^2(1-2\epsilon)^2.
\end{align*}
(iii) $\epsilon\in(0,\frac{1}{3})$
\begin{align*}
    &\int_{[0,1]}dx_1\int_{[x_1,(x_1+\epsilon)\wedge 1]}dx_2\int_{[x_2,(x_2+\epsilon)\wedge 1]}dx_3\int_{[x_3,(x_3+\epsilon)\wedge 1]}dx_4\\
    &=\int_{[0,1]}dx_1\left(\int_{[x_1,(x_1+\epsilon)\wedge 1]\cap[0,1-2\epsilon]}\epsilon^2 dx_2+\int_{[x_1,(x_1+\epsilon)\wedge 1]\cap[1-2\epsilon,1-\epsilon]}\epsilon(1-\epsilon-x_2)\right.\\
    &\left.-\left[\frac{(1-x_3)^2}{2}\right]_{1-\epsilon}^{x_2+\epsilon}dx_2+\int_{[x_1,(x_1+\epsilon)\wedge 1]\cap[1-\epsilon,1]}-\left[\frac{(1-x_3)^2}{2}\right]_{x_2}^1 dx_2\right)\\
    &=\int_{[0,1-3\epsilon]}dx_1\int_{[x_1,x_1+\epsilon]}\epsilon^2 dx_2
    +\int_{[1-3\epsilon,1-2\epsilon]}dx_1\left(\int_{[x_1,1-2\epsilon]}\epsilon^2 dx_2\right.\\
    &\left.+\int_{[1-2\epsilon,x_1+\epsilon]}\epsilon(1-\epsilon-x_2)+\frac{1}{2}[\epsilon^2-(1-\epsilon-x_2)^2]dx_2\right)\\
    &+\int_{[1-2\epsilon,1-\epsilon]}dx_1\left(\int_{[x_1,1-\epsilon]}\epsilon(1-\epsilon-x_2)+\frac{1}{2}[\epsilon^2-(1-\epsilon-x_2)^2]dx_2\right.\\
    &\left.+\frac{1}{2}\int_{[1-\epsilon,x_1+\epsilon]}(1-x_2)^2 dx_2\right)
    +\frac{1}{2}\int_{[1-\epsilon,1]}dx_1\int_{[x_1,1]}(1-x_2)^2 dx_2\\
    &=\int_{[0,1-3\epsilon]}\epsilon^3 dx_1+\int_{[1-3\epsilon,1-2\epsilon]}\epsilon^2(1-2\epsilon-x_1)-\left[\frac{\epsilon(1-\epsilon-x_2)^2}{2}\right]_{1-2\epsilon}^{x_1+\epsilon}\\
    &+\frac{1}{2}\left(\epsilon^2(x_1+3\epsilon-1)+\left[\frac{(1-\epsilon-x_2)^3}{3}\right]_{1-2\epsilon}^{x_1+\epsilon}\right)dx_1\\
    &+\int_{[1-2\epsilon,1-\epsilon]}-\left[\frac{\epsilon(1-\epsilon-x_2)^2}{2}\right]_{x_1}^{1-\epsilon}+\frac{1}{2}\left(\epsilon^2(1-\epsilon-x_1)+\left[\frac{(1-\epsilon-x_2)^3}{3}\right]_{x_1}^{1-\epsilon}\right)\\
    &-\frac{1}{2}\left[\frac{(1-x_2)^3}{3}\right]_{1-\epsilon}^{x_1+\epsilon} dx_1
    +\frac{1}{2}\int_{[1-\epsilon,1]}-\left[\frac{(1-x_2)^3}{3}\right]_{x_1}^1 dx_1\\
    &=\epsilon^3(1-3\epsilon)+\int_{[1-3\epsilon,1-2\epsilon]}\epsilon^2(1-2\epsilon-x_1)+\frac{1}{2}[\epsilon^3-\epsilon(1-2\epsilon-x_1)^2]\\
    &+\frac{1}{2}\epsilon^2(x_1+3\epsilon-1)-\frac{1}{6}[\epsilon^3-(1-2\epsilon-x_1)^3]dx_1\\
    &+\frac{1}{2}\int_{[1-2\epsilon,1-\epsilon]}\epsilon(1-\epsilon-x_1)^2+\epsilon^2(1-\epsilon-x_1)-\frac{(1-\epsilon-x_1)^3}{3}\\
    &+\frac{1}{3}[\epsilon^3-(1-\epsilon-x_1)^3]dx_1+\frac{1}{6}\int_{[1-\epsilon,1]}(1-x_1)^3 dx_1\\
    &=\epsilon^3(1-3\epsilon)-\left[\frac{\epsilon^2(1-2\epsilon-x_1)^2}{2}\right]_{1-3\epsilon}^{1-2\epsilon}+\frac{1}{2}\epsilon^4+\frac{1}{2}\left[\frac{\epsilon(1-2\epsilon-x_1)^3}{3}\right]_{1-3\epsilon}^{1-2\epsilon}\\
    &+\frac{1}{2}\left[\frac{\epsilon^2(x_1+3\epsilon-1)^2}{2}\right]_{1-3\epsilon}^{1-2\epsilon}-\frac{1}{6}\epsilon^4-\frac{1}{6}\left[\frac{(1-2\epsilon-x_1)^4}{4}\right]_{1-3\epsilon}^{1-2\epsilon}\\
    &\frac{1}{2}\left(-\left[\frac{\epsilon(1-\epsilon-x_1)^3}{3}\right]_{1-2\epsilon}^{1-\epsilon}-\left[\frac{\epsilon^2(1-\epsilon-x_2)^2}{2}\right]_{1-2\epsilon}^{1-\epsilon}+\frac{1}{3}\left[\frac{(1-\epsilon-x_1)^4}{4}\right]_{1-2\epsilon}^{1-\epsilon}\right.\\
    &\left.+\frac{1}{3}\epsilon^4+\frac{1}{3}\left[\frac{(1-\epsilon-x_1)^4}{4}\right]_{1-2\epsilon}^{1-\epsilon}\right)-\frac{1}{6}\left[\frac{(1-x_1)^4}{4}\right]_{1-\epsilon}^1\\
    &=\epsilon^3(1-3\epsilon)+\frac{1}{2}\epsilon^4+\frac{1}{2}\epsilon^4-\frac{1}{6}\epsilon^4+\frac{1}{4}\epsilon^4-\frac{1}{6}\epsilon^4+\frac{1}{24}\epsilon^4+\frac{1}{6}\epsilon^4+\frac{1}{4}\epsilon^4-\frac{1}{24}\epsilon^4\\
    &+\frac{1}{6}\epsilon^4-\frac{1}{24}\epsilon^4+\frac{1}{24}\epsilon^4\\
    &=\epsilon^3(1-3\epsilon)+\frac{3}{2}\epsilon^4\text{ so}\\
    &P(\mathscr{C})=24\epsilon^3(1-3\epsilon)+36\epsilon^4.
\end{align*}
For $n\geq 2,$
\begin{align*}
    &\int_{[0,1]}dx_1\int_{[x_1,x_1+\epsilon]\cap[0,1]}dx_n\left(\int_{[x_1,x_n]}dx_n\right)^{n-2}=\int_{[0,1]}dx_1\int_{[x_1,(x_1+\epsilon)\wedge 1]}(x_n-x_1)^{n-2}dx_n\\
    &=\int_{[0,1]}\frac{(x_n-x_1)^{n-1}}{n-1}|_{x_n=x_1}^{(x_1+\epsilon)\wedge 1} dx_1\\
    &=\frac{1}{n-1}\int_{[0,1]}\left[(x_1+\epsilon)\wedge 1-x_1\right]^{n-1}dx_1\\
    &=\frac{1}{n-1}\left\{\int_{[0,1-\epsilon]}\epsilon^{n-1}dx_1+\int_{[1-\epsilon,1]}(1-x_1)^{n-1}dx_1\right\}\\
    &=\frac{1}{n-1}\left\{\epsilon^{n-1}(1-\epsilon)-\left[\frac{(1-x_1)^n}{n}\right]_{1-\epsilon}^1\right\}=\frac{1}{n-1}\left\{\epsilon^{n-1}(1-\epsilon)+\frac{1}{n}\epsilon^n\right\}\\
    &=\frac{\epsilon^{n-1}}{n-1}(1-\epsilon+\frac{\epsilon}{n})=\frac{\epsilon^{n-1}}{n-1}\left[1-(1-\frac{1}{n})\epsilon\right]\text{ so}\\
    &P(\mathscr{G}(0)\text{ is }\epsilon\text{-trivial})=\epsilon^{n-1}\left[n-(n-1)\epsilon\right].
\end{align*}
Observe that $\epsilon^{n-1}\left[n-(n-1)\epsilon\right]=1$ for $n=1.$ So $$P(\mathscr{G}(0)\text{ is }\epsilon\text{-trivial})= \epsilon^{n-1}\left[n-(n-1)\epsilon\right]\text{ for }n\geq 1.$$
\begin{align*}
   &P(x_i(0)\in B(x_1(0),\epsilon/2)\text{ for all }i\in[n])=\int_{[0,1]}dx_1\left(\int_{[x_1-\frac{\epsilon}{2},x_1+\frac{\epsilon}{2}]\cap [0,1]}dx\right)^{n-1}\\
   &=\int_{[0,1]}\left[(x_1+\frac{\epsilon}{2})\wedge 1-(x_1-\frac{\epsilon}{2})\vee 0\right]^{n-1}dx_1\\
   &=\int_{[0,\frac{\epsilon}{2}]}(x_1+\frac{\epsilon}{2})^{n-1}dx_1+\int_{[\frac{\epsilon}{2},1-\frac{\epsilon}{2}]}[(x_1+\frac{\epsilon}{2})-(x_1-\frac{\epsilon}{2})]^{n-1}dx_1\\
   &+\int_{[1-\frac{\epsilon}{2},1]}[1-(x_1-\frac{\epsilon}{2})]^{n-1}\\
   &=\frac{(x_1+\frac{\epsilon}{2})^n}{n}|_0^{\frac{\epsilon}{2}}+\epsilon^{n-1}(1-\epsilon)-\left[\frac{(1+\frac{\epsilon}{2}-x_1)^n}{n}\right]_{1-\frac{\epsilon}{2}}^1\\
   &=\frac{1}{n}[\epsilon^n-(\frac{\epsilon}{2})^n]+\epsilon^{n-1}(1-\epsilon)+\frac{1}{n}[\epsilon^n-(\frac{\epsilon}{2})^n]\\
   &=\frac{2}{n}\epsilon^n(1-\frac{1}{2^n})+\epsilon^{n-1}(1-\epsilon)
\end{align*}
\end{proof}
In conclusion, the probability of consensus is positive and is bounded from below by the one that an initial profile is connected for $1\leq n\leq 4$. In particular for one dimension, the disconnected-preserving property for a profile engenders an upper bound $P(\mathscr{G}(0)\hbox{ is connected})$ for the probability of consensus, and therefore $P(\mathscr{C})=P(\mathscr{G}(0)\hbox{ is connected})$ for $1\leq n\leq 4$.


\begin{thebibliography}{99}
\bibitem{p2} 
\newblock S. R. Etesami and T. Ba{\c s}ar,
\newblock Game-theoretic analysis of the Hegselmann-Krause model for Opinion dynamics in finite dimensions,
\newblock \emph{IEEE Transactions on Automatic Control}, \textbf{60} (2015), 1886--1897.

\bibitem{p1}
\newblock H. Li, 
\newblock Mixed Hegselmann-Krause dynamics, 
\newblock https://arxiv.org/abs/2010.03050, 2020.

\bibitem{p3}
\newblock R. Hegselmann and U. Krause, 
\newblock Opinion dynamics and bounded confidence models, analysis, simulation,
\newblock Artif. Societies Social Simul., vol. 5,
pp. 1–33, 2002.

\bibitem{p4}
\newblock J. Lorenz,
\newblock Repeated Averaging and Bounded-Confidence, Modeling,
Analysis and Simulation of Continuous Opinion Dynamics,
\newblock Ph.D. dissertation, University of Bremen, Bremen, Germany, 2007.

\bibitem{p5}
\newblock C. Castellano, S Fortunato and V Loreto,
\newblock Statistical physics of social dynamics,
\newblock Rev. Modern Phys., 81 (2009), pp. 591-646

\end{thebibliography}
\end{document}